\documentclass[11pt]{amsart}
\usepackage{amsmath,amscd,amssymb,amsfonts,amsthm,hyperref}
\usepackage{tcolorbox}
\usepackage{tikz-cd}
\hypersetup{colorlinks,   linkcolor={red!50!black},    citecolor={blue!70!black},   urlcolor={blue!80!black}}

\usepackage[cmtip,matrix,arrow]{xy}
\newtheorem{theorem}{Theorem}[section]

\theoremstyle{definition}    
\newtheorem{definition}[theorem]{Definition}
\theoremstyle{remark}

\newtheorem{remark}[theorem]{Remark}

\newtheorem{examples}[theorem]{Examples}
\newcommand\A{\mathcal{A}}

\newcommand{\W}{\mathcal{W}}

\renewcommand{\L}{\mathcal{L}}
\renewcommand{\O}{\mathcal{O}}
\newcommand{\T}{\mathbb{T}}

\newcommand{\ca}{\mathcal}

\newcommand{\R}{\mathbb{R}}

\newcommand{\Z}{\mathbb{Z}}

\newcommand{\on}{\operatorname}

\newcommand{\Mult}{  \on{Mult}}

\newcommand\qu{/\kern-.7ex/} % Categorical quotients
\newcommand{\hra}{\hookrightarrow}

\renewcommand{\d}{{\mbox{d}}}

\newcommand{\f}{\frac}

\newcommand{\p}{\partial}

\newcommand{\eeq}{\end{eqnarray*}}
\newcommand{\beq}{\begin{eqnarray*}}

\newcommand{\wt}{\widetilde}
\newcommand{\mf}{\mathfrak}
\newcommand{\rra}{\rightrightarrows}
\newcommand{\sz}{\mathsf{s}}
\newcommand{\tz}{\mathsf{t}}

\newcommand{\I}{\mathcal{I}}
\renewcommand{\supset}{\supseteq}
\renewcommand{\subset}{\subseteq}

\begin{document}
\sloppy
\title{An introduction to weightings along submanifolds}
\author{Eckhard Meinrenken}

\maketitle

%\tableofcontents

\section{Introduction}
The concept of \emph{weighting} on a manifold was introduced by Richard Melrose \cite{mel:cor} under the name of \emph{quasi-homogeneous structures}, and by Yiannis Loizides and myself in \cite{loi:wei} under the present name. Further developments may be found in \cite{beh:the,goo:won,hud:lin,hud:mul,loi:lie}.  

The principal idea of weighting is to give a coordinate-free meaning to the process of assigning weights to coordinate functions. 
Constructions of this type appear in many area of mathematics, such as in the classification of singularities \cite{arn:cri}, 
algebraic geometry (e.g., weighted blow-ups), anisotropic asymptotic expansions (e.g., \cite{choi:priv1}), or the theory of hypoelliptic operators. One of our original motivations was to obtain a better understanding of the `nonlinear' tangent groupoids
arising in this context of manifolds with Lie filtrations \cite{choi:tan,erp:tan,hig:eul}. 

These notes are a write-up of a talk given at the Ghent Geometric Analysis Seminar in 2023, and I thank the organizers of the seminar, 
Michael Ruzhansky and David Santiago Gomez Cobos, for the kind invitation. 

%As a typical example,  the  Kolmogorov operator \[ P=\partial_y+x\partial_z-\partial_x^2\] is not elliptic but is hypo-elliptic: If $Pu$ is smooth then $u$ is smooth. Note that $P$  is weighted homogeneous of degree $-2$ after assigning weights $1,2,3$ to the coordinate functions $x,y,z$, and is `elliptic in a weighted sense'.  

\section{Definition and examples}
A weighting of a manifold $M$ along a submanifold $N\subset M$ is a filtration on the algebra of functions, 
\begin{equation}\label{eq:sheaf}
C^\infty(M)=C^\infty(M)_{(0)}\supset C^\infty(M)_{(1)}\supset \cdots \end{equation}
generalizing the usual order of vanishing of a functions along $N$. The conditions on such filtration to define a weighting can be expressed in various equivalent forms. 

\begin{definition}[Coordinate definition \cite{loi:wei,mel:cor}]\label{def:coordinate}
A \emph{weighting} of order $r$ on $M$, with weights $0\le w_1\le \ldots \le w_{\dim M}\le r$, is  a 
filtration \eqref{eq:sheaf} on the algebra of functions such that $M$ admits a covering by charts,
defining coordinates $x_a$,  with the property that $f\in C^\infty(M)_{(i)}$ if and only if the restriction of $f$ to any such chart domain lies 
in the ideal generated by monomials 
\[ x^s=\prod_a x_a^{s_a} \] 
of weight $\sum_a s_a w_a\ge i$. 
\end{definition}
We refer to the coordinates in this definition as weighted coordinates. Note that the definition does not make explicit reference to a submanifold; it is automatic that 
\[ C^\infty(M)_{(1)}=\I_N\]
is the vanishing ideal of a closed, embedded submanifold $N\subset M$. 
The weighting determines a filtration of the cotangent bundle, restricted to $N$, 
\begin{equation} \label{eq:filtrationcotangent} 
T^*M|_N=(T^*M|_N)_{(0)}\supset (T^*M|_N)_{(1)}\supset\cdots   \supset (T^*M|_N)_{(r)}\supset 0
\end{equation}
where the space of sections of $(T^*M|_N)_{(i)}$ is $C^\infty(N)$-linearly spanned by  all 
\begin{equation}\label{eq:spannedby}
\d f|_N,\ \ \ \ f\in C^\infty(M)_{(i)}
\end{equation}
 In particular, $(T^*M|_N)_{(1)}$ is the conormal bundle of $N$. 
%
%tangent bundle, \begin{equation} \label{eq:filtrationtangent} TM|_N= (TM|_N)_{(-r)}\supset \cdots (TM|_N)_{(0)}=TN\end{equation}where $(TM|_q)_{(-i)}$ for $q\in N$ consists of tangent vectors $v\in T_qM$ such that $v(f)=0$ for $f$ of filtration degree $>i$. Dually,  Note $(T^*M|_N)_{(i)}=\on{ann}((TM|_N)_{(-i+1)}$. 
%
In \cite{loi:wei}, it is shown that the following coordinate-free definition is equivalent to Definition \ref{def:coordinate}:

\begin{definition}[Intrinsic definition]\label{def:intrinsic}
A weighting of order $r$ on $M$  is given by a filtration 
\eqref{eq:sheaf} on the sheaf of functions, such that 
\begin{enumerate}
\item $C^\infty(M)_{(1)}$ is the vanishing ideal $\ca{I}_N$ of a closed submanifold $N\subset M$.
\item The differentials \eqref{eq:spannedby}  span subbundles 
\[ (T^*M|_N)_{(i)}\subset T^*M|_N,\] 
equal to $0$ for 
$i\ge r$. 
\item For $i\ge 2$, 
\[ C^\infty(M)_{(i)}\cap \ca{I}_N^2=\sum_{0< j<i}C^\infty(M)_{(j)}C^\infty(M)_{(i-j)}.\]
\end{enumerate}
\end{definition}	
Note that a filtration of order $r$ is uniquely determined by the filtration up to degree $i=r$; for $i>r$ we have 
$C^\infty(M)_{(i)}\subset \ca{I}^2$, which is described by the third condition.

\begin{examples}\label{ex:firstexamples}
\begin{enumerate}
	\item A weighting of order $r=1$ is the filtration of functions given by the 
	usual order of vanishing along 	
	a closed submanifold $N$. 
	\item\label{it:b} A weighting or order $r=2$ is equivalent to a submanifold $N\subset M$ together with a subbundle $Q\subset TM|_N$ 
	containing $TN$. Here $f$ has filtration degree $2$ if and only if it vanishes along $N$ and $\d f|_N$ vanishes on $Q$. 
	\item\label{it:c} 
	Suppose $M$ is a \emph{graded bundle} as defined by Grabowski-Rodkievicz \cite{gra:gra}. That is, $M$ 
	carries an action of the monoid $(\R,\cdot)$  by `scalar multiplications' 
	\[ \kappa_t\colon M\to M\] 
	with 
	$\kappa_{st}=\kappa_s\kappa_t$. The fixed point set of such an action is a closed submanifold $N\subset M$. 
	The manifold $M$ acquires a weighting along $N$, where a function 
	$f$ has filtration degree $i$ if and only if $\kappa_t^*f=O(t^i)$. 
	\item\label{it:d} 
	A collection of submanifolds $N_1,\ldots,N_r\subset M$  \emph{intersects cleanly} if for all $m\in M$, there exists a coordinate chart $U$ centered at $m$, with the property that all non-empty $N_i\cap U$ map to open subsets of coordinate subspaces. Let
	$C^\infty(M)_{(i_1,\ldots,i_r)}$ be the ideal of functions vanishing to order $i_1$ on $N_1$, to order $i_2$ on $N_2$, and so on. 
	Then we obtains a weighting of order $r$ along $N=N_1\cap\cdots\cap N_r$ by letting $C^\infty(M)_{(i)}$ be the sum over all 
	$C^\infty(M)_{(i_1,\ldots,i_r)}$ such that $i_1+\ldots+i_r=i$. More generally, one can define \emph{multi-weightings} for families of cleanly intersecting submanifolds; the `total weighting' is then a weighting for the intersection. 
\end{enumerate}	
\end{examples}	
\begin{remark}
The concept of weighting also applies in the holomorphic category, replacing the filtration on $C^\infty(M)$ with a filtration of the \emph{sheaf} $\O_M$ of holomorphic functions.  	
\end{remark}

\section{Weighted normal bundle}
Recall that the \emph{associated graded algebra} of a $\Z$-filtered algebra $\A$ (thus $\A_{(i)}\subset \A_{(i+1)}$ and 
$\A_{(i)}\A_{(j)}\subset \A_{(i+j)}$) is the graded algebra with summands 
\[ \on{gr}(\A)^i=\A_{(i)}/\A_{(i+1)}.\]
The \emph{Rees algebra} of $\A$ is the algebra 
\[ \on{Rees}(\A)\subset \A \otimes \R[u,u^{-1}]\]
consisting of 
Laurent polynomials  $\sum_{i\in \Z} a_i \otimes u^{-i}$ with coefficients $a_i\in \A_{(i)}$. For $a\in \A_{(i)}$, let 
$\wt{a}_{[i]}$ be the corresponding element $a\otimes u^{-i}$ of the Rees algebra, and 
$a_{[i]}\in  \on{gr}(\A)^i$ the image under the quotient map. There is a surjective morphism of graded algebras 
$\on{Rees}(\A)\to \on{gr}(\A)$, taking $\wt{a}_{[i]}$ to $a_{[i]}$.  
Applying these constructions to the filtration of $\A=C^\infty(M)$ given by a weighting (declaring that  $\A_{(i)}=C^\infty(M)$ for $i\le 0$), 
we make the following definitions (inspired by Haj-Higson \cite{hig:eul}): 

\begin{definition}\cite{loi:wei}
The \emph{weighted normal bundle} is the set  of morphisms of unital algebras, 
\[ \nu_\W(M,N)=\on{Hom}_{\on{alg}}(	\on{gr}(C^\infty(M)),\R).\]
The \emph{weighted deformation space} $\nu_W(M,N)$ is 	the set of morphisms of unital algebras, 
\[ \delta_\W(M,N)=\on{Hom}_{\on{alg}}(\on{Rees}(C^\infty(M)),\R).\]
\end{definition}
For $f\in C^\infty(M)_{(i)}$, the elements $f_{[i]},\ \wt{f}_{[i]}$ may be regarded as functions on the weighted normal bundle and weighted deformation space, respectively. As a special case, the constant function $1\in C^\infty(M)$, regarded as having filtration degree $-1$, 
 defines a function $\wt{1}_{[-1]}=1\otimes u$ on the deformation space; we will denote this function by  
\begin{equation}\label{eq:pi}
\pi \colon \delta_\W(M,N)\to \R.\end{equation}
The surjective algebra morphism $\on{Rees}(C^\infty(M))\to \on{gr}(C^\infty(M))$ dualizes to an injective map 
\begin{equation}\label{eq:u0}
\ \nu_\W(M,N)\hra \delta_\W(M,N);
\end{equation}
$f_{[i]}$ is the restriction of $\wt{f}_{[i]}$ under this map. For $c\neq 0$, the algebra morphism 
$ \on{Rees}(C^\infty(M))\to C^\infty(M)$ taking $\sum_i f_i\otimes u^i$ to $\sum c^i f_i$ dualizes to an inclusion 
\begin{equation}\label{eq:uc}
M\hra \delta_\W(M,N).
\end{equation}

\begin{theorem}\cite{loi:wei}. The weighted normal bundle $\nu_\W(M,N)$ is a manifold, in such a way that 
the functions $f_{[i]}$ for $f\in C^\infty(M)_{(i)}$ are smooth.  Similarly,  $\delta_\W(M,N)$ is a manifold, in such a way that 
the functions $\wt{f}_{[i]}$ for $f\in C^\infty(M)_{(i)}$ are smooth. The map \eqref{eq:pi} is a surjective submersion, and 
\eqref{eq:u0} and \eqref{eq:uc} give  identifications 
\[ \pi^{-1}(c)=\begin{cases}
\nu_\W(M,N) & c=0,\\
M& c\neq 0.
\end{cases}\]
The group $\R^\times$ acts smoothly on the deformation space by the \emph{zoom action} $t\mapsto \kappa_t$: Here $t\in \R^\times$ acts on 
$\nu_\W(M,N)$ by fiberwise scalar multiplication, and on $\pi^{-1}(\R^\times)=M\times \R^\times$ by $\kappa_t(x,c)\mapsto (x,t^{-1}c)$. 
\end{theorem}
By construction, if $f$ has filtration degree $i$, then the function $\wt{f}_{[i]}$ is homogeneous of degree $i$ for the zoom action: 
\[ \kappa_t^* \wt{f}_{[i]}=t^i \wt{f}_{[i]}.\]
Weighted coordinates $x_a$ on $M$, of weights $w_a$, gives rise to coordinates 
${x_a}_{[w_a]}$ on the weighted normal  bundle. For the deformation space, the functions $\wt{x_a}_{[w_a]}$ together with 
the map $\pi$ (regarded as a function, denoted by a variable `$u$') are coordinates on the deformation space.  

The weighted deformation space is a natural setting for `weighted' linearization problems. Consider for example the filtration of 
on vector fields, 
\[ \mf{X}(M)=\mf{X}(M)_{(-r)}\supset \mf{X}(M)_{(-r+1)}\supset \cdots\]
where a vector field $X$ has filtration degree $i\ge -r$ if the Lie derivative $\L_X$ raises the filtration degree of functions by $i$. 
In this case, $X$ a \emph{homogeneous}  approximation $X_{[i]}$, which is a vector field on $\nu_\W(M,N)$ of homogeneity $i$ 
characterized by its property 
\[ \L_{X_{[i]}}f_{[j]}=(\L_X f)_{[i+j]}.\] 
A similar prescription defines a vector field $\wt{X}_{[i]}$ on the deformation space. One finds that $\wt{X}_{[i]}$ 
 is vertical with respect to the submersion $\pi$, and may be regarded as a family of vector fields on the fibers $\pi^{-1}(c)$, equal to $X_{[i]}$ for $c=0$ and to $c^i X$ for $c\neq 0$. 

See \cite{me:eul} for concrete applications of this set-up to normal form theorems. As another application, one obtains an intrinsic definition of 
\emph{weighted (real)  blow-ups}, along the lines of Debord-Skandalis \cite{deb:blo}.  Let 
\[ \on{S}_\W(M,N)=(\nu_\W(M,N)-N)/\R_{>0}\to N\]
be the weighted sphere bundle. 
\begin{theorem}\cite{loi:wei}
The disjoint union 	
\[ \on{Bl}_\W(M,N)\cong    \on{S}_\W(M,N) \sqcup (M-N)\]
is a manifold with boundary, in such a way that the blow-down map $ \on{Bl}(M,N)\to M$ (given by projection to $N$ of the sphere
bundle and inclusion of $M-N$) is smooth. 
\end{theorem}
The weighted blow-up is constructed as the quotient of the manifold with boundary 
$(\delta_\W(M,N)-(N\times \R))\cap \pi^{-1}(\R_{\ge 0})$ under the zoom action of $\R_{>0}$; the proof amounts to 
showing that the latter action is free and proper.

\section{Lie filtrations}
Many examples of \emph{hypo-elliptic} operators are obtained from \emph{Lie filtrations} on manifolds, given by a 
filtration of the tangent bundle 
\[ TM=H_{(-r)}\supseteq H_{(-r+1)}\supseteq\cdots \supseteq H_{(0)}=0\]
%\[ 0=H_{-1}\subseteq \cdots \subseteq H_{-r}=TM;\]
such that the resulting filtration on vector fields is compatible with brackets 
$[\Gamma(H_{(i)}),\Gamma(H_{(j)})]\subseteq \Gamma(H_{(i+j)})$. 
Given such a Lie filtration, the associated graded bundle 
\[ \on{gr} TM=\bigoplus_{i=-r}^{-1} H_{(i)}/H_{(i+1)}\]
is a family of nilpotent Lie algebras (not necessarily a locally trivial Lie algebra bundle). It exponentiates to a family of nilpotent Lie groups,  called the \emph{osculating groupoid} \cite{erp:tan} 
or \emph{tangent cone} \cite{choi:tan},
\begin{equation}\label{eq:osculating} T_H M\to M.\end{equation}
Lie filtrations give rise to weightings in  natural ways. Suppose that $N\subset M$ is a filtered submanifold as defined by Haj-Higson \cite{hig:eul}: that is, the intersections $H_{(i)}\cap TN$ are smooth subbundles.  
As shown in \cite{loi:lie}, one obtains a weighting of order $r$ along $N$ by setting 
\[C^\infty(M)_{(0)}=C^\infty(M),\ C^\infty(M)_{(1)}=\I_N\] 
and
\[ C^\infty(M)_{(i)}=\{f|\ \forall X\in \Gamma(H_{(j)})\colon \L_X f\in C^\infty(M)_{(i+j)}\}\] 
for $i>1$. The weighted normal bundle $\nu_W(M,N)\to N$ and weighted deformation space $\delta_\W(M,N)$ are exactly  the normal bundle and deformation space of the filtered submanifold, as defined in \cite{hig:eul}. In particular, the osculating groupoid is 
\begin{equation}\label{eq:oscu}
 T_HM=\nu_\W(M\times M,M_\Delta)
\end{equation}
and it comes with a version of Connes' tangent groupoid \cite{con:non} 
\begin{equation}\label{eq:connes}
 \T_H M=\delta_\W(M\times M,M_\Delta).\end{equation}
The groupoid structure of $\T_H M$ over $M\times \R$ is obtained from the groupoid structure on $\pi^{-1}(0)=T_HM$ (as a bundle of nilpotent groups) and the pair groupoid structure on $\pi^{-1}(c)=M\times M$ for $c\neq 0$. More conceptually, it follows from the functorial properties of the weighted deformation space construction -- see the following section. 
This groupoid $\T_H M$ was used by  van-Erp and Yuncken \cite{erp:gro} (building on ideas of Debord and Skandalis \cite{deb:adi})   
to develop a pseudo-differential calculus for manifolds with Lie filtrations.  

\begin{remark}
There are examples of hypoelliptic operators that do not correspond to Lie filtrations, but instead 
are related to more complicated 
`singular' Lie filtrations. See in particular the recent work of Androulidakis-Mohsen-Yuncken \cite{and:conv,and:pse}. In 
\cite{loi:lie}, it is shown that singular Lie filtrations once again define weightings along suitable `clean' submanifolds. 
\end{remark}

%\begin{remark}	For $i=-r+1,\ldots,0$ and $X\in \mf{X}(M)$, let  $X^{(i)}\in \mf{X}(T_{r-1}M)$ be its lift of homogeneity $i$, defined by $X^{(i)}f^{(j)}=(\L_X f)^{(i+j)}$ for all $f$. Given a Lie filtration, the lifts $X^{(i)}$ for $X\in \Gamma(H_{(i)})$ span an integrable distribution in $T_{-r+1}M$, hence define a foliation.     	\end{remark}

\section{Multiplicative weightings}
The groupoid structures on \eqref{eq:oscu} and \eqref{eq:connes} may be understood from the weightings perspective: The weighting on  $M\times M$ is compatible with the pair groupoid structure, the diagonal is a weighted subgroupoid, and for this reason the resulting weighted normal bundle and deformation space acquire groupoid structure. To explain this in more detail, we need to consider weightings on Lie groupoids. 

Suppose $G\rra M$ is a Lie groupoid. Denote by $\sz,\tz\colon G\to M$ the source and target maps, and by $\Mult_G\colon G^{(2)}\to G$ 
the groupoid multiplication, defined on the set of \emph{composable arrows} (i.e., pairs $(g,h)\in G\times G$ such that $\sz(g)=\tz(h)$). 
What does it mean for a weighting on $G$ of order $r$ to be compatible with the groupoid structure? 
In \cite{loi:wei}, using the description of weightings in terms of submanifolds $Q\subset T_{r-1}G$ as in 
Section \ref{sec:higher} below, a weighting is called multiplicative if $Q$ is a Lie subgroupoid of 
$T_{r-1}G\rra T_{r-1}M$ . This definition, while short, is somewhat difficult to verify in practice. A more direct approach  was developed in the Ph.D. thesis of Dan Hudson \cite{hud:mul}. 
\begin{theorem}\cite{hud:mul} A weighting on a Lie groupoid $G\rra M$ is multiplicative if and only if 
\begin{enumerate}
	\item the source and target maps $\sz,\tz\colon G\to M$ are \emph{weighted submersions}, 
	\item the set of units $M\to G$ is a \emph{weighted submanifold},
	\item the groupoid multiplication $\Mult_G\colon G^{(2)}\to G$ is a weighted submersion.   
\end{enumerate}	
The submanifold of $G$ defined by the weighting is a Lie subgroupoid $H\rra N$. 
\end{theorem}
The theorem requires some explanation. To begin, a \emph{weighted submanifold} of a weighted manifold $M$ is a submanifold with the property that there exists weighted coordinates for $M$ that are also submanifold coordinates. 
Such a submanifold is then a weighted manifold in its own right. Similarly, a \emph{weighted submersion} from a weighted manifold $M$ 
to another weighted manifold $M'$ is a submersion $f\colon M\to M'$ such that for each $p\in M$ there exists weighted coordinates around $p$ and $f(p)$, respectively,  that are also  submersion coordinates. (These conditions also admit coordinate-free formulations \cite{hud:mul}.) 
Under the assumption that $\sz,\tz\colon G\to M$ are weighted submersions, it is shown in \cite{hud:mul} that $G^{(2)}$ is a weighted submanifold of $G\times G$, and hence the condition on $\Mult_G$ being a weighted submersion make sense.

In joint work with Dan Hudson, we also found the following alternative version of this result. Given a weighting of $G\rra M$ along a 
subgroupoid $H\rra N$, we obtain a filtration of $T^*G|_H$, and hence dually of $TG|_H$, where $TG_{(-i)}$ is the 
annihilator of $(T^*G|_H)_{(i+1)}$.

\begin{theorem}[Hudson-M] 
A weighting of $G\rra M$ along $H\rra N$ is multiplicative if and only if the units $M\subset G$ are a weighted submanifold, 
the graph of the groupoid multiplication $\on{Gr}(\Mult_G)\subset G\times G\times G$ is a weighted submanifold, and 
each subbundle
\[ (TG|_H)_{(-i)}\subset TG\]
is a Lie subgroupoid of the tangent groupoid $TG\rra TM$. 	
\end{theorem}

Given a multiplicative weighting of $G\rra M$ along $H\rra N$, it is shown in \cite{hud:mul} that 
\[ \nu_\W(G,H)\rra \nu_\W(M,N),\ \ \delta_\W(G,H)\rra \delta_W(M,N)\]
are Lie groupoids. 

\begin{remark}
The infinitesimal counterpart to multiplicative weightings on Lie groupoids are the infinitesimally multiplicative weightings on Lie 
algebroids. Again, these may be defined in terms of higher tangent bundles (Section \ref{sec:higher}; see also \cite{loi:wei}). A more 
direct description was obtained in \cite{hud:mul}, which also establishes the compatibility of the 
(weighted) normal bundle and deformation space functors.
\end{remark}

\begin{remark}
At this stage, one should  develop a pseudo-differential calculus for groupoids with multiplicative weightings, along the lines of 
van Erp and Yuncken \cite[Section 12]{erp:gro}, see also Debord-Skandalis \cite{deb:lie}.  
We hope to return to this problem in future work. 	
\end{remark}

\section{Appendix: Weightings in terms of jet bundle} \label{sec:higher}
According to Example \ref{ex:firstexamples}\eqref{it:b}, a weighting of order $r=2$ is described by a submanifold $Q\subset T_1M$. 
This suggests a generalization to weightings of arbitrary order, in terms of a sequence of submanifolds of higher tangent bundles. 
Recall \cite{kol:nat} that the $k$-th order tangent bundle 
\[ T_kM\to M\]
is the fiber bundle whose elements are the  $k$-jets $j^k_0(\gamma)$ of curves $\gamma\colon \R\to M$. 
%The higher tangent bundles fit into a tower \begin{equation}\label{eq:jettower} \cdots \to  T_iM\to T_{i-1}M\to \cdots \to T_0M=M,\end{equation}with quotient maps $\pi_i\colon T_iM\to T_{i-1}M$ taking $i$-jets to $i-1$-jets. 
The scalar multiplication of the monoid $(\R,\cdot)$ on $\R$ defines an action on curves $\gamma$, which descends to a
monoid action $t\mapsto \kappa_t$ (as in Example \ref{ex:firstexamples}\eqref{it:d}) on $T_kM$. The fixed point set  for this action is $M\subset T_kM$, embedded as jets of constant paths. 

%Jets of constant paths define an embedding $M\subset T_kM$ as a submanifold.  

Every $f\in C^\infty(M)$ gives rise to a sequence of lifts $f^{(j)}\in C^\infty(T_kM)$ for $j\le k$, by
\[ f^{(j)}( j^r_0(\gamma))=\f{1}{j!} \f{d^j}{\d t^j}|_{t=0} f(\gamma(t)).\]
%The function $f^{(j)}$ on the various bundles $T_iM$ are related by pullbacks, hence we don't indicate $i$ in the notation. 
Given local coordinates $x_a$ on $U\subset M$, the lifts $x_a^{(j)}$ with $0\le j\le k$ define coordinates on $T_kM|_U$. 

Suppose now that the filtration \eqref{eq:sheaf} defines an order $r$  weighting of $M$ along $N$. Let 
\[ Q\subset T_{r-1}M\] be the subset cut out by lifts $f^{(i-1)}$ of functions $f$ of filtration degree 
$i\le r$. 
%Then \eqref{eq:jettower} restrict to projections \begin{equation}\label{eq:ntower}\cdots N_i\to N_{i-1}\to \cdots N_0=N\end{equation}
%for a sequence of submanifolds $N_i\subset T_iM$. Here  \[ N_i\subset \pi_i^{-1}(N_{i-1})\] is the subset cut out from the pre-image $\pi_i^{-1}(N_{i-1})$ by the vanishing of $f^{(i)}$ for $f$ of filtration degree $i+1$.  The fact that the $N_i$ are indeed submanifolds may be verified in weighted coordinates for $M$. Note that the $N_i$ are $\kappa_t$-invariant, and that the fixed point set under the monoid action is equal to $N$. Furthermore, $N_i=\pi_i^{-1}(N_{i-1})$ for  $i>r$. Hence, the entire sequence is determined by the submanifold $N_r$, by taking images of pre-images under the $\pi_i$. 
It is shown in  \cite{loi:wei} that $Q$ is a submanifold, and that the weighting may be recovered from $Q$ by letting 
\begin{equation}\label{eq:iweight}
C^\infty(M)_{(i)}=\{f\in C^\infty(M)|\ f^{(i-1)}|_Q=0\}\end{equation}
for $i\le r$. 

This raises the question of which submanifolds $Q\subset T_{r-1}M$ define weightings. A rather complicated answer was given in \cite{loi:wei}; a simpler criterion was recently  found by Aaron Gootjes-Dreesbach \cite{goo:won}. 

\begin{theorem} (See \cite[page 16]{goo:won}.) 
	A closed submanifold $Q\subset T_{r-1}M$ defines a weighting of order $r$  if and only if:
	\begin{enumerate}
		\item\label{it:ta} The vanishing ideal $\I_Q$ is generated by functions of the form $f^{(i)}$.
		\item\label{it:tb} If $f^{(i)}\in \I_Q$ then also $f^{(j)}\in\I_Q$  for all $j\le i$. 
	\end{enumerate}
\end{theorem}
\begin{proof}
	Suppose $Q\subset T_{r-1}M$ satisfies these properties.  Consider the direct sum decomposition
	\begin{equation}\label{eq:cotangentgrading}
	T^*(T_{r-1}M)|_M=\bigoplus_{j=0}^{r-1} (T^*(T_{r-1} M)|_M)^j
	\end{equation}
	where the $j$-th summand is spanned by $\d f^{(j)}|_M$ for $f\in C^\infty(M)$.  In local coordinates $x_a$ on $M$, 
	\[ \d f^{(j)}|_M=\sum_a \f{\p f}{\p x_a} \d x_a^{(j)}|_M;\]
	this shows that there are isomorphisms 
	\begin{equation}\label{eq:caniso}
	(T^*(T_{r-1} M)|_M)^j\to T^*M
	\end{equation}
	taking $\d f^{(j)}|_M$ to $\d f$. The the conormal bundle to $Q$ along $N=Q\cap M$, 
	\begin{equation}
	\on{ann}(TQ)|_N\subset T^*(T_{r-1}M)|_M
	\end{equation}
	is spanned by differentials $\d g|_N$ for functions $g\in C^\infty(T_{r-1}M)$ such that $g|_Q=0$. 
	Condition \eqref{it:ta} shows that it is already spanned by differentials of  functions of the form $f^{(j)}$; hence 
	$\on{ann}(TQ)|_N$ is a graded subbundle:
	\[ \on{ann}(TQ)|_N=\bigoplus_{j=0}^{r-1} (\on{ann}(TQ)|_N)^j.\]
	Let $(T^*M|_N)_{(j+1)}$ be the image of its $j$-th summand under \eqref{eq:caniso}, and put $(T^*M|_N)_{(0)}=
	T^*M|_N$. Condition \eqref{it:tb} shows that this defines a filtration \eqref{eq:filtrationcotangent} of $T^*M|_N$. 
	Let 
	\[ k_i=n-\on{rank}(T^*M|_N)_{(i)}\]
	so that $\dim N= k_0\le \cdots \le k_r=n=\dim M$. To show that the filtration 
	\eqref{eq:iweight} gives a weighting, we construct weighted coordinates 
	near any given $p\in N\subset M$. 
	Choose functions  
	\[ x_a,\ a>k_{r-1},\] 
	defined near $p$, such that the lifts $x_a^{(r-1)}$ vanish on $Q$ and such that 
	\[ \d x_a^{(r-1)}|_p,\ a>k_{r-1}\] 
	are a basis of  $(\on{ann}(TQ)|_p)^{r-1}$. By \eqref{it:tb} the functions 
	$x_a^{(r-2)}$ vanish on $Q$ as well, and  $\d x_a^{(r-2)}|_p,\ a>k_{r-1}$ are linearly independent. 
	Extend to a collection of functions 
	\[ x_a,\ a>k_{r-2}\] such that  $\d   x_a^{(r-2)}|_N,\ a>k_{r-2}$ are a basis of 
	$(\on{ann}(TQ)|_p)^{r-2}$. Proceeding in this manner, we eventually obtain a collection of functions $x_a$, for all $a>k_0$, such that for all $j$, the lifts 
	$\d   x_a^{(j)},\ a>k_{j-1}$ are  a basis of $(\on{ann}(TQ)|_p)^{j}$. The differentials $\d x_a|_p,\ a>k_0$ are linearly independent, and we may choose  additional functions $x_a,\ a\le k_0$ so that the set of all $\d x_a|_p$ is a basis of  $T^*M|_p$. These functions define coordinates on a 
	sufficiently small open neighborhood 
	$U$ of $p$, and the submanifold of $T_{r-1}M|_U$ corresponding to the weighting is defined by the vanishing of $x_a^{(j)}$ 
	for $a>k_{j-1}$. But this is exactly $Q\cap T_{r-1}M|_U$. 
\end{proof}

\bibliographystyle{amsplain} 
%\bibliography{../../../Biblio/ref}

\def\cprime{$'$} \def\polhk#1{\setbox0=\hbox{#1}{\ooalign{\hidewidth
			\lower1.5ex\hbox{`}\hidewidth\crcr\unhbox0}}} \def\cprime{$'$}
\def\cprime{$'$} \def\cprime{$'$} \def\cprime{$'$} \def\cprime{$'$}
\def\polhk#1{\setbox0=\hbox{#1}{\ooalign{\hidewidth
			\lower1.5ex\hbox{`}\hidewidth\crcr\unhbox0}}} \def\cprime{$'$}
\def\cprime{$'$} \def\cprime{$'$} \def\cprime{$'$} \def\cprime{$'$}
\providecommand{\bysame}{\leavevmode\hbox to3em{\hrulefill}\thinspace}
\providecommand{\MR}{\relax\ifhmode\unskip\space\fi MR }
\providecommand{\MRhref}[2]{%
	\href{http://www.ams.org/mathscinet-getitem?mr=#1}{#2}
}
\providecommand{\href}[2]{#2}

\end{document}